\documentclass[12pt]{amsart}
\usepackage{amsmath, amsfonts, latexsym, array, amssymb, }
\usepackage{psfrag}
\usepackage[all]{xypic}
\usepackage{graphicx}
\usepackage{color}

\newtheorem{thm}{Theorem}[section]
\newtheorem{cor}[thm]{Corollary}
\newtheorem{lem}[thm]{Lemma}

\newtheorem{prop}[thm]{Proposition}
\newtheorem{cla}[thm]{Claim}
\newtheorem{defn}[thm]{Definition}
\newtheorem{rem}[thm]{Remark}

\newtheorem{que}{Question}

\newcommand{\modified}[1]{#1~} 
\newcommand{\deleted}[1]{}

\newcommand{\comment}[1]{}

\newcommand{\new}[1]{{\bf #1 }}

\newcommand\diam{{\operatorname{diam}}}
\newcommand\Diff{{\operatorname{Diff}}}
\newcommand\eps{\epsilon}
\newcommand\erg{{\operatorname{erg}}}
\newcommand\FF{{\mathcal F}}

\newcommand\loc{{\operatorname{loc}}}
\newcommand\Prob{{\operatorname{Prob}}}
\newcommand\RR{{\mathbb R}}

\renewcommand\top{{\operatorname{top}}}
\newcommand\TT{{\mathbb T}}
\newcommand\ZZ{{\mathbb Z}}

\begin{document}

\title[Entropic Stability] {Entropic Stability Beyond Partial Hyperbolicity}
\author{J\'er\^ome Buzzi and Todd Fisher}
\address{C.N.R.S. \& D\'epartement de Math\'ematiques, Universit\'e Paris-Sud, 91405 Orsay, France}
\address{Department of Mathematics, Brigham Young University, Provo, UT 84602}
\email{jerome.buzzi@math.u-psud.fr}
\email{tfisher@math.byu.edu}
\thanks{}

\date{March 23, 2009}
\commby{}

\begin{abstract}
We analyze a class of deformations of Anosov diffeomorphisms: these
$C^0$-small, but $C^1$-macroscopic deformations break the topological conjugacy class but leave the high entropy dynamics unchanged. More precisely, there is a partial conjugacy between the deformation and the original Anosov system that identifies all invariant probability measures with entropy close to the maximum. We also establish expansiveness around those measures.

This class of deformations contains many of the known nonhyperbolic robustly transitive diffeomorphisms. In particular, we show that it includes a class of nonpartially hyperbolic, robustly transitive diffeomorphisms described by Bonatti and Viana.
 \end{abstract}

\maketitle

\section{Introduction}

Observing that a physical system is only  known up to some finite precision, Andronov and Pontriaguyn \cite{Andronov} suggested in 1937 that the study of dynamical systems should focus on stable systems, i.e., those that do not change under small perturbations. Rather strikingly, it has turned out that the topologically $C^1$-stable dynamics (also called structurally stable systems) can be analyzed;
indeed, they are exactly the uniformly hyperbolic systems~\cite{Mane88} that satisfy an additional technical assumption. However, the structurally stable diffeomorphisms are not dense 
in the $C^1$-topology and therefore the study of such systems is insufficient even up to an arbitrarily small perturbation. 
Hence, much of the focus of current research in dynamical systems is to extend
our understanding beyond uniform hyperbolicity, in the hope of eventually obtaining 
a global theory of ``most" systems (see for example \cite{Palis}).

A natural approach is to consider weaker forms of stability. First, one should 
design a stability property that holds for interesting examples outside of uniform hyperbolicity. 
Second, one should establish large sets of stable dynamics by finding robust
global phenomena enforcing this property. Third, if this stability fails to hold densely,
one should look for  robust local mechanisms responsible for its failure.
\footnote{This dichotomy of phenomena and mechanisms has been put forward by Pujals
\cite{Pujals} and is related to Palis conjectures for a global picture of dynamics \cite{Palis}.}

The obvious candidate for such a notion is $C^r$-stability with $r>1$, but it turns out to be very difficult to study because of a lack of perturbation lemmas. We keep the $C^1$-topology but replace topological conjugacy by a looser, entropy-based notion and realize the first step of the above program.

\subsection{Stability of the large entropy measures for Bonatti-Viana diffeomorphisms}

We will show that many constructions, including the Bonatti-Viana diffeomorphisms which we will describe below, have strong stability properties, though they are not hyperbolic and therefore not structurally stable.

We first review some standard definitions.  Let $\Diff^1(M)$ denote the space of $C^1$-diffeomorphisms of a compact manifold $M$ endowed with some Riemannian structure. $\Diff^1(M)$ is endowed with its usual topology and distance $d_{C^1}(f,g)$ 
(see \cite[Section 8.1.1]{RobinsonBook}).
A map $f:M\rightarrow M$  is {\it transitive} if there exists some $x\in M$ whose forward orbit is dense in $M$.  A diffeomorphism $f:M\rightarrow M$ is  {\it $C^1$-robustly transitive} if there exists a neighborhood $\mathcal{U}$ of $f$ in $\mathrm{Diff}^1(M)$ such that each $g\in\mathcal{U}$ is transitive.
A {\it Borel isomorphism} of $\psi:X\to Y$ is a bijection such that $\psi$ and $\psi^{-1}$ are both Borel maps. Let $\Prob(f)$ be the set of invariant Borel probability
measures of $f$. 

The topological entropy of a system $(X,f)$, denoted $h_{\mathrm{top}}(f)$, is a number that measures the topological complexity of the system.  On the other hand, if $\mu\in\Prob(f)$, then the measure theoretic entropy, denoted $h_{\mu}(f)$, of a dynamical system is a number that measures the complexity of the system as seen by the measure $\mu$.   (See~\cite[Sections 3.1 and 4.3]{KH}  for precise definitions.)
The variational principle  states that if $f$ is a continuous self-map of a compact metrizable space, then $h_{\mathrm{top}}(f)=\sup_{\mu\in\Prob(f)}h_{\mu}(f)$, see for instance~\cite[p. 181]{KH}.  A measure $\mu\in\Prob(f)$ such that $h_{\mathrm{top}}(f)=h_{\mu}(f)$ is a {\it measure of maximal entropy}.  If there is a unique measure of maximal entropy, then $f$ is called {\it intrinsically ergodic}.

\begin{defn}\label{def:h-stable}
A diffeomorphism $f:M\to M$ is \new{entropically $C^r$-stable} if
for every $g\in \Diff^r(M)$ that is $C^r$-close to $f$ there exists
a  Borel isomorphism $\psi:M'\to M''$ where $M', M''$ are Borel subsets of $M$ such that the following properties hold:
 \begin{itemize}
   \item $\psi\circ g = f\circ \psi$ on $M'$,
   \item $\tilde h(f,M\setminus M'') := \sup \{ h(f,\nu) :
\mu\in\Prob(f|_{M\setminus M''})\}< h_\top(f)$, and
   \item $\tilde h(g,M\setminus M') < h_\top(g)$.
  \end{itemize}
\end{defn}

It is convenient to call \emph{large entropy measures,} the \emph{ergodic} invariant probability measures with entropy greater than some constant $h$ (strictly less than the topological entropy). We denote their set by $\Prob_\erg^h(f)$.  Hence,  a system is entropically $C^r$-stable if its large entropy measures ``stay the same" for any sufficiently close diffeomorphism.  

We shall establish that large entropy measures also maintain the following property, enjoyed by the Anosov system (for which we can take $X_1=X$):

\begin{defn}
Let $f:X\to X$ be a Borel isomorphism of a metric space $X$ and $\mathcal M$ be a set of invariant probability measures of $f$.
We say that $\mathcal M$ is {\bf almost expansive} if there exists a number $\eps_0>0$  such that, for every $\mu\in\mathcal M$, for $\mu$-a.e. $x\in M$, for all $y\in M$ $\sup_{n\in\ZZ} d(f^nx,f^ny)\geq \eps_0$ implies $x=y$.
\end{defn}

What can we say about the low entropy measures? On the one hand,  a recent result of Hochman \cite{Hochman} shows that entropy-conjugacy with an Anosov system implies Borel conjugacy to this Anosov system up to subsets of zero measure with respect to any aperiodic invariant  probability measure. This gives the corresponding strengthening of the stability property.  On the other hand,  expansivity can fail dramatically around low entropy measures. The theory of symbolic extension and entropy structures of Boyle and Downarowicz~\cite{BD, Dow} allows to precisely define such phenomena. 

Recall that a symbolic system is defined as the left shift $(x_n)_{n\in\ZZ}\mapsto
(x_{n+1})_{n\in\ZZ}$ acting on a shift invariant, compact subset of $\mathbb N^{\mathbb Z}$
(not necessarily of finite type).
Given a homeomorphism $f$ of a compact metrizable space $X$  a 
{\it symbolic extension} of $(X,f)$  is a continuous surjection 
$\varphi\colon  \Sigma\to X$ such that $f\circ\varphi=\varphi\circ \sigma$ 
and $(\Sigma,\sigma)$ is a symbolic system.    Boyle and Downarowicz~\cite{BD} have shown that for a given dynamical system there is a natural connection between how entropy arises on finer and finer scales "around" an invariant probability measure and the existence of such symbolic extensions that approximate, closely in entropy, the system equipped with that measure.  We shall see that, for some examples of entropically stable systems considered in this paper,  the dynamics at low entropy does not only fail to be expansive but prevents the existence of any symbolic extension.

We will  study diffeomorphisms $f:M\to M$ of a compact manifold $M$ with a weak form of hyperbolicity called a dominated splitting.  
A $Df$-invariant splitting of the tangent bundle of some invariant set $\Lambda$
$$
T_\Lambda M=E_1\oplus\cdots\oplus E_k
$$
is \emph{dominated} if each bundle has constant dimension (at least two of them non-zero) and there exists an integer $\ell\geq1$ with the following property. For every $x\in\Lambda$, all $i=1,\dots,(k-1)$,  and every pair of unitary vectors $u\in E_1(x)\oplus\dots\oplus E_i$ and $v\in E_{i+1}(x)\oplus\dots\oplus E_k(x)$,  it holds that
$$
 \frac{|Df_x^\ell (u)|}{|Df^{\ell}_x (v)|}\leq \frac{1}{2}.
$$
(See for example~\cite[Appendix B, Section 1]{BUH} for properties of systems with a dominated splitting.)  A diffeomorphism $f\in\mathrm{Diff}^1(M)$ is {\it partially hyperbolic} if there exists a dominated splitting $TM=E^s \oplus E^c\oplus E^u$ where $E^s$ is uniformly contracting, $E^u$ is uniformly expanding (at least one of $E^s$ and $E^u$ is non-trivial).  The diffeomorphism $f$ is {\it strongly partially hyperbolic} if $E^s$ and $E^u$ are both non-trivial.  A strongly partially hyperbolic diffeomorphism is {\it hyperbolic} (or Anosov) if $TM=E^s\oplus E^u$.  The (stable) index is the dimension $\dim E^s$.

The first goal of the present paper is to analyze  a class diffeomorphisms described by Bonatti and Viana~\cite{BV}.   This was the first example of a robustly transitive diffeomorphism that is not partially hyperbolic.  

\begin{thm}\label{thm:applyToBV}
There exist a $C^\infty$-diffeomorphism, $f$, of the $4$-torus, $\TT^4$, and an open set $\mathcal{U}\subset \mathrm{Diff}^1(M)$ containing $f$ such that each $g\in \mathcal{U}$  is $C^1$-robustly transitive, not partially hyperbolic and are all entropy-conjugate to the same  Anosov diffeomorphism.

Moreover, for each $g\in\mathcal U$, the large entropy measures are almost expansive. Nevertheless, there exist a non-empty open set $\mathcal{V}\subset \mathcal{U}$ and a $C^1$-residual set $\mathcal{D}\subset \mathcal{V}$ such that each $g\in \mathcal{D}$ has no symbolic extension.
\end{thm}

\begin{cor}
All diffeomorphisms $g\in\mathcal U$ above are not partially hyperbolic, are not structurally stable or $\Omega$-stable but are $C^1$-entropically stable. 
\end{cor}

In particular, this gives a nonempty open set of nonpartially hyperbolic diffeomorphisms
with constant topological entropy and unique measures of maximal entropy that define pairwise isomorphic measure-preserving transformations.  We derive this from an abstract result (Theorem \ref{thm:main}).

\begin{rem}
The diffeomorphisms $g\in\mathcal D$ defined above are far from entropy-expansive (i.e., they do not satisfy $h_\loc(f,\eps)=0$ for any $\eps>0$). Indeed, asymptotic entropy-expansivity, i.e.,  $\lim_{\eps\to0} h_\loc(g,\eps)=0$ would imply the existence of a symbolic extension with nice properties (see Sec. \ref{sec:non-concentration} for definitions). Nevertheless, these diffeomorphisms satisfy:
 $$h_\top(g)=h_\top(g,\eps_0)$$
for $\eps_0>0$ the implicit constant in the definition of almost expansivity (which can even be chosen independently of $g\in\mathcal V$).
\end{rem}

\subsection{Previous Results}\label{sec:previous}

Newhouse and Young \cite{NY83} proved that a class of partially hyperbolic, nonhyperbolic robustly transitive diffeomorphisms\footnote{This class was described originally by Shub~\cite{shub}.} that are $C^0$ deformations of Anosov diffeomorphisms are $C^1$-entropically stable.   More precisely, they  showed that these diffeomorphisms have a unique measure of maximal entropy that is isomorphic to the measure of maximal entropy for the Anosov system but this can be strengthened to the entropic stability defined above.

Together with M. Sambarino and V. Vasquez the authors of the present work proved the following related result:

\begin{thm}~\cite{bfsv}~\label{t.bsfv} For any $d\geq 3$, there exists a nonempty open set $\mathcal{U}$
in $\mathrm{Diff}(\mathbb T^d)$ satisfying:
 \begin{itemize}
   \item each $f\in \mathcal{U}$ is strongly partially hyperbolic, robustly transitive, and 
   $C^1$-entropically stable;
in particular the topological entropy is locally constant at $f$;
  \item each $f\in \mathcal{U}$ has equidistributed periodic points; and
 \item no $f\in \mathcal{U}$ is Anosov or structurally stable.
 \end{itemize}
\end{thm}

We note that the abstract result below (Theorem \ref{thm:main}) gives another proof  of the first point of the above theorem.

\medbreak

Using different techniques, F.~Rodriguez Hertz, J.~Rodriguez Hertz, Tahzibi, and Ures~\cite{RHRHTU} have obtained a rather precise description of the following systems:

\begin{thm}\cite{RHRHTU}
Consider an accessible partially hyperbolic diffeomorphism of a 3-dimensional manifold having compact center leaves.  Either it has a unique entropy maximizing measure with zero center Lyapunov exponent,
or it has a positive, even number of entropy maximizing ergodic, invariant measures, all of them with nonzero center Lyapunov exponent.
\end{thm}

We refer to their paper for definitions of the terms above. Notice that this shows that transitivity, entropic stability and non-uniqueness of the entropy maximizing measures can robustly coexist.  Also, it shows that uniqueness of the entropy maximizing measure does not hold generically outside of the uniformly hyperbolic systems, even assuming topological transitivity.

\subsection{Further Questions}
There are a number of natural questions that follow from Theorem~\ref{thm:applyToBV}.  

\subsubsection{A phenomenon for entropic stability}

After this analysis of a large class of examples, we would like to find the phenomena responsible for entropic stability and in particular formulate more general conditions, instead of specifying a perturbative scheme. The following condition might be sufficient:

\begin{defn} \cite{BuzziICMP}
Let $f\in\Diff(M)$.  Define
 $$
    h^{k}(f):=\sup\{h_\top(f,\phi([0,1]^{k})): \phi\in C^\infty(\RR^k,M)\}.
 $$
We say that  a dominated splitting $TM=E^{cs}\oplus E^{cu}$ is \new{entropy-hyperbolic} if the following holds:
 \begin{itemize}
  \item $h^{\dim E^{cu}-1}(f)< h_\top(f)$, and
  \item $h^{\dim E^{cs}-1}(f^{-1}) < h_\top(f)$.
  \end{itemize}
 \end{defn}

\begin{que}
Does the entropy-hyperbolicity of a dominated splitting imply the finiteness of the number of entropy maximizing ergodic invariant measures? Does it implies  entropic-stability?
\end{que}

We think that the answer to the first question above is affirmative, on the basis of partial results assuming stronger versions of the condition. 

\begin{rem}
Entropy-hyperbolicity, as formulated above, seems to be excessively restrictive. For instance, the disjoint union of two Anosov systems with distinct indices cannot satisfy it for trivial reasons.  Also, we  do not know of a system with an entropy-hyperbolic dominated splitting which is not given by an isotopy from an Anosov system. 
\end{rem}

\subsubsection{Entropy-conjugacy to uniform systems}

The stability observed by Newhouse and Young and in this work, follows from a weak form of conjugacy  to a uniformly hyperbolic system: almost conjugacy (topological conjugacy up to negligible sets for the maximal entropy measure of each system) or entropy-conjugacy (see below Def. \ref{def:h-conjugate}). One might think that such a conjugacy is actually the rule rather than the exception. They formulated this idea as follows:

\begin{que}\cite{NY83}\label{q.ny1} For any compact manifold $M$ and $r\geq1$, let $\mathcal{B}(M)$ denote the set of $C^r$ diffeomorphisms such that
\begin{enumerate}
\item $f$ has finitely many ergodic, invariant measures of maximal entropy, and
\item on the support of each such measure, $f$ is almost conjugate to some Axiom A diffeomorphism.
\end{enumerate}
Is $\mathcal{B}(M)$ residual in $\mathrm{Diff}^r(M)$?
\end{que}

One can of course reformulate this question, replacing the above notion of almost conjugacy by entropy-conjugacy (almost conjugacy does not imply and is not implied by entropy-conjugacy).

\subsubsection{Generic stability and finiteness}

\begin{que}\label{q.ny} Does a generic diffeomorphism admit finitely many entropy-maximizing ergodic and invariant measures? Is it entropically stable? 
\end{que}

However, for all we know, entropic stability (like structural stability) could fail to be dense. More precisely,
S. Crovisier~\cite{Cro} suggested that the  diffeomorphisms with homoclinic 
classes robustly without a dominated splitting might provide a (large) set of points of
variation of the topological entropy and thus give a negative answer to Question \ref{q.ny}.

\section{Abstract Result}\label{s.abstract}

In this section,  we state Theorem \ref{thm:main}, from which Theorem~\ref{thm:applyToBV} will be deduced in section \ref{s.BV}. Theorem \ref{thm:main} is our main result. It proves that a large class of deformations of Anosov
diffeomorphisms that are big in the $C^1$ topology, do not modify large entropy measures. We first state the somewhat technical assumptions as three definitions and then the theorem. We conclude this section with an outline of the proof.

For simplicity, we assume that $M=\mathbb T^d$, the $d$-dimensional torus and leave the obvious modifications necessary to deal with the (slightly) more general manifolds carrying Anosov systems to the interested reader. Also, Anosov or Anosov system will mean Anosov diffeomorphism.

\medbreak

The first requirement will ensure that the perturbation is $C^1$-small, except possibly on a union of a given number of well-separated balls of small radius:

\begin{defn}\label{def:sparse}
An \new{$(\eps,N)$-sparse deformation} (or just: $(\eps,N)$-deformation) of $f\in\Diff^1(M)$ is a diffeomorphism $g:M\rightarrow M$
such that there exist $x_1,\dots,x_N\in M$
and $r>0$ satisfying:
 \begin{itemize}
  \item $d_{C^1}(g|_{M\setminus B_r},f|_{M\setminus B_r})<\eps$ where $B_r:=\bigcup_{i=1}^N B(x_i,r)$;
  \item $r<\eps$;
  \item $\min_{i\ne j} d(x_i,x_j)>\eps^{1/2}$.
 \end{itemize}
 $B_r$ is called a \new{strong support} of the deformation and $r$ is called its \new{radius}.
\end{defn}

Before stating the second requirement, we need to recall some facts about cone conditions and hyperbolicity.  The {\it cone} $C^1_\alpha$ of aperture $\alpha>0$ defined by a decomposition $E^1\oplus E^2$ of a Euclidean space $E$ is:
 $$
  C^1_\alpha:=\{ v^1+v^2\in E: v^i\in E^i\text{ and }\|v^2\|\leq\alpha\|v^1\| \}.
 $$
$\dim(E^1)$ is called the dimension of the cone. For a manifold $M$, a {\it cone field} is the specification of a cone  $C(x)$ of fixed dimension in each $T_xM$, $x\in M$.  A boundaryless submanifold $\Sigma$ is {\it tangent to a cone field} $C$ if 
\begin{enumerate}
\item[(i)] the submanifold and the cone have the same (constant)
dimension and 
\item[(ii)] $T_x\Sigma\subset C(x)$ at every $x\in M$.
\end{enumerate}

Let $f$ be Anosov with an adapted Riemannian metric: there exists a $Df$-invariant continuous splitting $TM=E^s\oplus E^u$   with the following bounds:
$$\begin{array}{rlll}
\lambda_0& =\min_{x\in M}\min_{v\in E^u_x\setminus\{0\}}\frac{\| Df v\|}{\| v\|}>1 \textrm{ and }\\
\mu_0& =\max_{x\in M}\max_{v\in E^s_x\setminus\{0\}}\frac{\| Df v\|}{\| v\|}<1.
\end{array}$$
($\lambda_0$ is the minimum expansion for $Df$ in the unstable direction, and $\mu_0$ is the minimum contraction for $Df$ in the stable direction).  The hyperbolicity strength is:
 $$
     \lambda_1=\min\{\lambda_0, \mu_0^{-1}\}>1.
 $$

 Let $C^u_\alpha$ and $C^s_\alpha$ denote the cones defined by the hyperbolic splitting of $f$ associated to $E^s\oplus E^u$ as above,  for an aperture $\alpha>0$ to be determined.

\begin{rem} Observe that the cones above are those defined by $f$, not by $\tilde f$. These will be the only cone fields that we consider. 
\end{rem}


We now formulate our second requirement on the deformations. It keeps the dominated splitting, even inside the strong support.

\begin{defn}A diffeomorphism $g:M\rightarrow M$ 
\new{($\alpha, \rho, \Lambda$)-respects the domination of $f$} if it satisfies the following for all $x\in M$ and all $y,z\in B(x,\rho)$ such that:
 $$
    y-x\in C^u_{\alpha}(x) \text{ and }g(z)-g(x)\in C^s_{\alpha}(g(x))
 $$
then
\begin{enumerate}
  \item $\|g(y)-g(x)\|/\|y-x\|>  \Lambda \|g(z)-g(x)\|/\|x-z\|$
  \item $g(y)-g(x)\in C^u_{\alpha}(g(x))$ and $z-x\in C^s_{\alpha}(x)$ 
 \end{enumerate}
 \end{defn}

 

This assumption of non-linear domination will ensure that large center-unstable disks are mapped by $g$ to similar disks and will be used to build invariant center-unstable foliations (and likewise for center-stable ones). The point of the above definition is to make the scale $\rho>0$ explicit.


Our third (and last) requirement is that even if a vector in the center-unstable direction can be contracted, this contraction is weak  (and analogously in the center-stable direction):

\begin{defn}
For $\gamma>0$ a diffeomorphism $g:M\rightarrow M$  is \new{$\gamma$-nearly hyperbolic} with respect to a dominated splitting
$TM=E^{cu}\oplus E^{cs}$ if for some $C\in(1, \infty)$ and for all $n\geq0$ the following conditions are satisfied:
\begin{enumerate}
\item[(i)] $\|Dg^n v^{cu}\|\geq C^{-1}e^{-\gamma n}$  for all $ v^{cu}\in E^{cu}$ and
\item[(ii)]    
     $ \|Dg^n v^{cs}\|\leq C e^{\gamma n}$ for all  $v^{cs}\in E^{cs}$.
  \end{enumerate}
\end{defn}

We will be interested in systems that are $\gamma$-nearly hyperbolic for $\gamma$ near zero.

\medbreak

The following notion introduced in \cite{BuzziSIM} describes the type of conjugacy we will obtain. Essentially the dynamics are conjugate with respect to ergodic invariant measures with large entropy:

\begin{defn}\label{def:h-conjugate}
Two dynamical systems $f:X\to X$ and $g:Y\to Y$ are \new{entropy-conjugate}
if there exists a partially defined bimeasurable bijection: $\psi:Y\setminus Y_0
\to X\setminus X_0$ such that:
 \begin{itemize}
   \item $\psi\circ g = f\circ \psi$ on $Y\setminus Y_0$,
   \item $\tilde h(f,X_0) := \sup \{ h(f,\nu) :
     \mu\in\Prob(f,X_0)\}< h_{\mathrm{top}}(f)$, and
   \item $\tilde h(g,Y_0) < h_{\mathrm{top}}(g)$.
  \end{itemize}
\end{defn}

We shall prove that the systems we consider are entropy conjugate to Anosov systems.  

\begin{rem}\label{r.conjugate}
Entropic stability of $f$ means that 
any diffeomorphism $C^1$-close to $f$ is entropy-conjugate to $f$.
\end{rem}

\medbreak

We now can state our main result. Theorem~\ref{thm:applyToBV} will follow directly from this result.

\begin{thm}\label{thm:main}
Let $f:M\to M$ be an Anosov diffeomorphism on $M=\TT^d$, $d\geq2$ and let $N\geq1$ be some integer.  There exists $t:=t(f,N)>0$ with the following property. 
Let $\eps,\alpha,\gamma \in(0,t)$ and let $\rho:= (\epsilon^{1/2}-2\epsilon)\cdot\mathrm{diam}M$ and $\Lambda>\frac{\epsilon^{1/2} + 2\epsilon}{\epsilon^{1/2}-2\epsilon}>1$.

Any $g\in\Diff^1(M)$ satisfying:
 \begin{enumerate}
  \item[(H1)] $g$ is an $(\eps,N)$-sparse deformation of $f$;
  \item[(H2)] $g$ ($\alpha, \rho,\Lambda$)-respects the domination of $f$;
  \item[(H3)] $g$ is $\gamma$-hyperbolic;
 \end{enumerate}
is entropy-conjugate to $f$. Moreover, the set of large entropy measures, $\Prob_\erg^h(g)$ for some $h<h_\top(g)$, is almost expansive.
\end{thm}

\subsection{Strategy of  proof}

The proof of Theorem~\ref{thm:main} splits
into the following steps:
 \begin{enumerate}
  \item  Existence of canonical, invariant, center-stable and
  center-unstable foliations for $g$.  The existence of these will follow from the dominated splitting
  and the respect of the domination of $f$ at a certain scale (Sec.~\ref{sec:foliations}).
  \item  Under  the factor map on the Anosov dynamics defined by the shadowing property, the measure-theoretic entropy can decrease only slightly  (Sec.~\ref{sec:extensionAnosov}).

  \item  The large entropy measures of $g$ give little mass to the strong support of the deformation (Sec.~\ref{sec:non-concentration}).
  \item The factor map is actually an entropy-conjugacy, proving  Theorem \ref{thm:main} (Sec.~\ref{sec:factor}).
 \end{enumerate}


\section{Invariant foliations}\label{sec:foliations}

The goal of this section is to build center-stable and center-unstable invariant foliations for our deformation $g$ of some Anosov system $f\in\Diff^1(M)$. As above, we restrict ourselves for simplicity to the  case where $M=\TT^d$.
We first recall some definitions.

\medbreak

A \emph{continuous foliation}  $\FF$ of dimension $k$ with $C^r$ leaves is a partition of the manifold such that there is locally an homeomorphism  mapping $\FF$ to the partition of $\mathbb R^d$ into $k$-planes for some $0\leq k\leq d$ and such that its restriction to any such plane is $C^r$.\footnote{It is well-known that even in the hyperbolic case,  the local homeomorphisms mapping the stable (or unstable) leaves to planes cannot always be chosen $C^1$.}

It is well-known that, in full generality, the existence of a dominated splitting does \emph{not} imply the existence of invariant foliations tangent to each sub-bundle.  In fact even with the stronger assumption of partial hyperbolicity the center direction may not be integrable. In~\cite{BW, P} there are discussions on the integrability of the bundles and some classical examples are given where the integrability does not hold.

Let $D$ be a smooth open disk embedded in $M$. Its \emph{inner radius} at some point $x\in D$, is the distance between $x$ and $\partial D:=\overline{D}\setminus D$.

\begin{thm}\label{thm:foliations}
Let $f\in\Diff^1(M)$ be Anosov with hyperbolic strength $\lambda>1$.  Let $1<\Lambda<\lambda$ and $\alpha>0$. There exists $\eps_1(f,\Lambda,\alpha)>0$  such that for all $0<\eps<\eps_1$ the following holds. 

 Let  $N\geq1$. and set $\rho=(\eps^{1/2}-2\epsilon)$. 
Let $g\in\Diff^1(M)$ be an $(\eps,N)$-deformation of $f$ which  $(\alpha, \rho,\Lambda)$-respects the domination of $f$.

Then $g$ has a dominated splitting $TM=E^{cs}\oplus E^{cu}$ with the same index as the hyperbolic splitting of $f$. Moreover, $g$ admits a center-stable foliation, $\mathcal F^{cs}$, and a center-unstable foliation, $\mathcal F^{cu}$, with the following properties:
 \begin{enumerate}
  \item each foliation is continuous with $C^1$-leaves;
  \item the leaves of the foliations are everywhere tangent to $E^{cs}$, $E^{cu}$ respectively; 
  \item the foliations are invariant under $g$.
 \end{enumerate}
\end{thm}

The following non-shrinking property is key to our construction:

 \begin{cla}\label{lem:domination}
 Let $f,g\in\Diff^1(M)$ and $N\geq 1$, $\Lambda>1,\alpha>0$, $\eps>0$, and $\rho:=(\eps^{1/2}-2\epsilon)$. Assume that $g$ is an $(\eps,N)$-deformation of an Anosov $f$ which $(\alpha,\rho,\Lambda)$-respects the domination of $f$. Assume also $\Lambda> \frac{\epsilon^{1/2} +2\epsilon}{\epsilon^{1/2}-2\epsilon}>1$.

 Let $x\in M$. Then  for any disk $D^u_{\rho}$ tangent to $C^u_{\alpha}$ and with inner radius $\rho$ at $x$, $g(D^u_{\rho})$ contains a disk tangent to $C^u_{\alpha}$ and with inner radius at least $\rho$ at $g(x)$.
 \end{cla}

\begin{proof}[Proof of Claim \ref{lem:domination}]
 Let $x\in M$. The invariance of the center-unstable cone implies that $g(D^u_\rho(x))$ is tangent to $C^u_\alpha$. It remains to see that this disk also has large diameter. As $f$ and $g$ are $C^0$-close, 
$d(f(y),g(y))\leq 2\epsilon\cdot\mathrm{diam}M.$
It follows that
 $$
 \begin{array}{llll}
      d(g(x),\partial g( D^u_\rho(x))) &\geq d(f(x),\partial f(D^u_\rho(x)))-4\epsilon\\
&       \geq \Lambda\rho-4\epsilon\\
& >\frac{\epsilon^{1/2}+2\epsilon}{\epsilon^{1/2}- 2\epsilon}(\epsilon^{1/2}-2\epsilon)-4\epsilon\\
&=\rho.
\end{array}$$
The claim is proved.
\end{proof}
\medskip

\begin{proof}[Proof of Theorem \ref{thm:foliations}]
We observe that the existence of the dominated splitting $E^{cs}\oplus E^{cu}$ for $g$ is a well-known consequence of the existence of the invariant cones.  See for instance~\cite[p. 293]{BUH}.

We first fix $x\in M$ and construct a sequence of disks of inner radius at least $\rho$ at $x$ and everywhere tangent to \modified{$C^{cu}$.

For each $n\geq0$, let $D_{x,n}^{-n}$ be an embedded open
smooth disk tangent to the  unstable cone field $C^u_\alpha$ defined by the hyperbolic splitting of $f$ and with inner radius $\rho$ at $g^{-n}(x)$.   By compactness of $M$, there is $\eps_0>0$, independent of $x$  so that, for any $0<\eps<\eps_0$, such a disk always locally exists, independently of any integrability condition.\footnote{For instance, we can take a disk in the local stable manifold of $x$ with respect to the Anosov diffeomorphism, $f$.} By reducing $\eps_0>0$ if necessary, we also ensure $\Lambda> \frac{\epsilon_0^{1/2} +2\epsilon_0}{\epsilon_0^{1/2}-2\epsilon_0} >1$. Using an obvious identification, we can chose $D_{x,n}^{-n}$ to be the graph of a map defined on an open subset of $E^u_\alpha(x)$ taking values in $E^{cs}_\alpha(x)$ and with small Lipschitz constant.}

Let $D_{x,n}^{k+1}:=g(D_{x,n}^k)\cap B(g^{k+1}x,\rho)$ for all
$k=-n,\dots,-1$. Standard graph transform estimates show that 
$(D_{x,n}^0)_{n\geq0}$ is a family of graphs of equicontinuous functions.
Moreover, their domains of definition contain the disk of 
$x+E^{cu}_x$ with center $x$ and radius $\rho$ from Claim~\ref{lem:domination}.   Thus, one can find a subsequence such that these functions converge uniformly to a function with bounded Lipschitz constant. Let $D_x$ be the limit graph. 

Note that every $y\in D_x$ satisfies $d(g^{-n}y,g^{-n}x)\leq \rho$ for all $n\geq0$. This allows the use the non-linear domination and to get, through standard arguments, that $D_x$ is $C^1$
with tangent spaces obtained as intersections of nested and
exponentially shrinking cones. In particular, these tangent
spaces coincide with $E^{cu}$.

Let us show that the sequence $(D_{x,n}^0)_{n\geq0}$ is actually
convergent by checking that the limit graph is unique
and independent of the choice of $D_{x,n}^{-n}$:

\begin{figure}[htb]
\begin{center}
\psfrag{x}{$x$}
\psfrag{y}{$y$}
\psfrag{z}{$y'$}
\psfrag{d}{$D_1$}
\psfrag{e}{$D_2$}
\psfrag{f}{$\Delta$}
\psfrag{g}{$C^{s}$}
\psfrag{h}{$C^{u}$}
\includegraphics{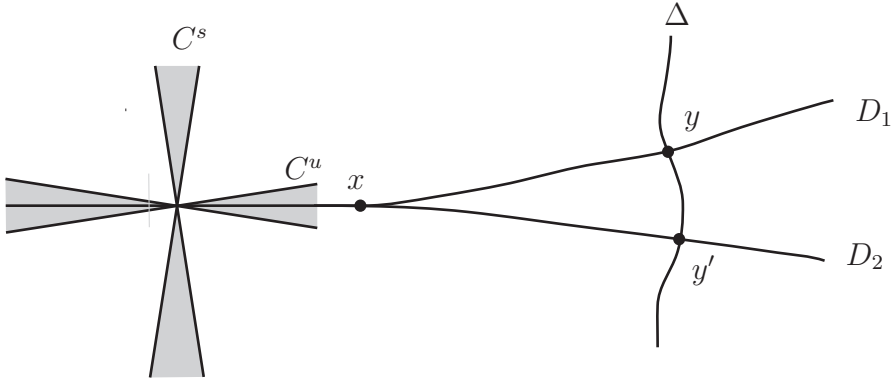}
\caption{unique disks}\label{f.linear}
\end{center}
\end{figure}
By contradiction, consider two distinct limit disks $D_1$ and $D_2$.
Thus, there exists a disk, $\Delta$, tangent to $E^{s}$ which intersects
$D_1$ and $D_2$ in two distinct points $y$ and $y'$.  By construction, for all $n\geq0$ we have
\begin{itemize}
\item $g^{-n}y, g^{-n}y'\in B(g^{-n}x,\rho)$;  
\item $g^{-n}y-g^{-n}x\in C^{cu}_\alpha(g^{-n}x)$; 
\item $g^{-n}y-g^{-n}y'\in C^{cs}_\alpha(g^{-n}x)$. 
\end{itemize}
It follows that, for all $n\geq0$:
 $$
     \frac{\|y-x\|}{\|g^{-n}y-g^{-n}x\|} \geq \Lambda^{n} \frac{\|y-y'\|}{\|g^{-n}y-g^{-n}y'\|}.
 $$
But $g^{-n}D,g^{-n}D'$ are contained in the cone $C^{u}_{\alpha_1}(g^{-n}x)$, thus
$$\|g^{-n}y-g^{-n}y'\|\leq K\|g^{-n}y-g^{-n}x\|$$ for some uniform $1<K<\infty$
and
 $$
     \|y-y'\| \leq \Lambda^{-n} \frac{\|g^{-n}y-g^{-n}y'\|}{\|g^{-n}y-g^{-n}x\|} \|y-x\|
       \leq K\Lambda^{-n} \|y-x\|.
 $$
Letting $n\to\infty$, we see that $y=y'$, a contradiction.

Note that the canonical character of the disks $D_x$ imply their equivariance: $g(D_x)\cap B(gx,\rho)=D_{gx}$. Also, the above argument implies the following uniqueness property.
For any $x,y\in M$,
if $z\in D^{cu}_x\cap D^{cu}_y$, then 
 \begin{equation}\label{eq:Dcu-uniq}
   D^{cu}_x\cap D^{cu}_y\cap B(z,\rho)\subset D^{cu}_z.
 \end{equation}

We now define the partition candidate to be an invariant center-unstable foliation. For each $x\in M$ we let $\FF^{cu}(x)$ be the set of all  $y\in M$ such that
there exist finitely many points $x_1,\dots,x_n$ satisfying: $x\in D^{cu}_{x_1}$, $y\in D^{cu}_{x_n}$ and $D^{cu}_{x_i}\cap D^{cu}_{x_{i+1}}\ne\emptyset$ for $i=1,\dots,n-1$. It follows from this definition that $\FF^{cu}$ is a partition and that it is invariant: $g(\FF^{cu}(x))=\FF^{cu}(g(x))$.

To prove that $\FF^{cu}$ is indeed a foliation, it remains to check that each $\FF^{cu}(x)$ intersects any small ball in a disjoint union of smooth disks and that the connected component of $x$ depends continuously in the $C^1$ topology of the base point $x$.
Let us set $F_x:=\FF^{cu}(x)\cap B(x,\rho/2)$. Obviously,
$$F_x=\bigcup_{y\in F_x} D^{cu}_y\cap B(x,\rho/2).$$ 
It follows from \eqref{eq:Dcu-uniq} that this is a disjoint union in the sense that either 
$$D^{cu}_y\cap B(x,\rho/2)=D^{cu}_{y'}\cap B(x,\rho/2)$$
or the two sets are disjoint. Thus, the connected component of $F_x$ containing $x$ is $D^{cu}_x\cap B(x,\rho/2)$. The construction of $D^{cu}_x$ shows that this is indeed a $C^1$ submanifold that depends continuously on $x$.

The claims of the theorem for $\FF^{cu}$ are proved. The proofs for $\FF^{cs}$ are completely analogous.
\end{proof}

\section{Almost Principal Extension of the Anosov}\label{sec:extensionAnosov}

In this section we let $g\in\mathrm{Diff}^1(M)$ ($M=\TT^d$), a sufficiently small $C^0$-perturbation of an Anosov diffeomorphism, $f$, and study the continuous factor map $\pi:(M,g)\to (M,f)$ given by the shadowing lemma (see Lemma \ref{l.shadowing} below).
We observe that the fibers $\pi^{-1}(x)$ for $x\in M$ have a small diameter. Second, we show that if $g$ respects the domination of $f$  and is nearly hyperbolic, then for a.e. $x\in M$, $\pi^{-1}(x)$ is contained in a leaf of the center-unstable or center-stable foliation (given by Theorem~\ref{thm:foliations}). 

\subsection{Shadowing for Anosov diffeomorphisms}

We recall the following well-known fact about hyperbolic dynamics.  For a proof see for instance~\cite[p. 109]{shub}.

\begin{lem}[Shadowing Lemma]\label{l.shadowing}
Let $f:M\to M$ be Anosov. There exist numbers $\eps_0>0$ and $K_0<\infty$ with the following property.

For any homeomorphism $g:M\to M$ with 
$$d_{C^0}(f,g):=\sup_{x\in M} d(f(x),g(x))+d(f^{-1}(x),g^{-1}(x))<\eps_0,$$ 
there is a  topological factor map $\pi:(M,g)\to(M,f)$. Moreover, $\sup_{x\in M} \diam(\pi^{-1}(x))\leq K_0 d_{C^0}(f,g)$. 
\end{lem}


To fix some notations, we recall the following classical notion.

\begin{defn}
Two foliations $\FF^{1},\FF^{2}$ have a \new{(local) product structure} if there exist constants $\tau_1,\tau_2>0$ and $1<K<\infty$ such that the following hold:
for all points $x,y$ within distance less than $\tau_1$, $\FF^{1}_{\tau_2}(x)$
(the connected component of $\FF^{1}(x)\cap B(x,\tau_2)$ containing
$x$) intersects the similarly defined $\FF^{2}_{\tau_2}(x)$ at
exactly one point, $z$, and $d(x,z)\leq K d(x,y)$.
\end{defn}

\begin{rem}A compact manifold with transverse continuous foliations $\mathcal{F}_1$ and $\mathcal{F}_2$ with $C^1$ leaves has a product structure for $\mathcal{F}_1$ and $\mathcal{F}_2$ for \emph{some} constants $\tau_1,\tau_2,K$.
\end{rem}

\subsection{Inclusion in $\FF^{cu}$ or $\FF^{cs}$}

The next proposition shows that for an appropriate deformation $g$, the  fibers of the ergodic invariant probability measures for $g$ disintegrated over $f$ are contained in the leaves of one of the dynamical foliations. 

\begin{prop}\label{prop:flat}
Let $f:M\to M$ be Anosov with shadowing constants $\eps_0>0$ and $K_0<\infty$ and hyperbolicity strength $\lambda>1$. Let $N\geq1$ and $1<\Lambda<\lambda$, $\alpha>0$,  $\tau_1,\tau_2>0$ and $K<\infty$. There exists $\eps_2(f,\Lambda,\alpha,\eps_0,K_0,\tau_1,\tau_2,K)>0$ with the following property for all $0<\eps<\eps_2$ and $g\in\Diff^1(M)$ which
 \begin{itemize}
  \item is an ($\eps,N$)-sparse deformation of $f$;
  \item ($\alpha,\rho,\Lambda)$-respects the  domination with $\rho:=(\eps^{1/2}-2\eps)$;
  \item preserves center-stable and center-unstable foliations $\FF^{cs},\FF^{cu}$ tangent to the cone fields $C^s_\alpha,C^u_\alpha$ and define a product structure with constants $\tau_1,\tau_2,K$.
 \end{itemize}
For any any $g$-invariant, ergodic probability measure $\nu$, there exists $\sigma=cs$ or $cu$ such that, 
 $$
         \text{ for $\nu$-a.e. $x\in M$ } \nu_x(\FF^\sigma_{\tau_2}(x)) = 1,
 $$
where $\nu=\int_M \nu_x \, d\pi_*\nu$ is the Rokhlin disintegration of $\nu$
w.r.t. $\pi$ (see \cite{Rudolph}).
\end{prop}


\begin{proof}
We shall establish the required property under a (finite) number of upperbounds on $\eps_2$. Recall the number $\eps_1(f,\Lambda,\alpha)>0$ from Theorem \ref{thm:foliations}. The first bound is:
 $$
     \eps_2<\min\{\eps_1, 1/({2K_0})^{2}\}
 $$
so $2\epsilon K_0<\epsilon^{1/2}$: the  balls of radius $K_0\epsilon$ around the $N$ centers of the $(\epsilon, N)$-deformation $g$ are disjoint.  This will be useful with regards to the Shadowing Lemma: recall that $\diam(\pi^{-1}(x))\leq K_0\eps$.

Let $\nu$ be an invariant ergodic measure for $g$ with its disintegration $(\nu_x)_{x\in M}$ as above. Since $\pi$ is a semi-conjugacy,  $\pi_*\nu$ is an ergodic probability measure for $f$.

It is convenient to set aside the trivial case where $\nu_x=\delta_x$ for a
set of positive (and hence full) $\pi_*\nu$-measure of points $x\in M$.

Let $\mu$ be the Cartesian square of $\nu$ relatively to the $\pi$ factor.
In other words, $\mu$ is the probability measure for $g\times g$ on $M\times M$
given by
 $$
   \mu = \int_M \nu_x\times\nu_x \, d\pi_*\nu.
 $$ 
 Observe that it is $g\times g$-invariant.

We define  a measurable function $R:M\times M\to [0,\infty]$ as follows.
For $(x,y)\in M\times M$, let 
$z$ be the unique intersection point
of $\FF^{cs}_{\tau_2}(x)$ and $\FF^{cu}_{\tau_2}(y)$ (if one exists). We set
 $$
     R(x,y) := \left\{
     \begin{array}{llll}
     \frac{d_{\FF^{cs}(x)}(x,z)}{d_{\FF^{cu}(y)}(y,z)} &\textrm{ if }z\textrm{ exists and }z\neq y\\
     \infty & \textrm{ else}
     \end{array}
     \right.
 $$
where $d_N(\cdot,\cdot)$ denotes the geodesic distance along the
submanifold $N$ using the induced Riemannian structure. 

Note that as $\pi(x)=\pi(y)$ for $\mu$-a.e.~$(x,y)$,  
$d(x,y)<K_0\epsilon$. To use the product structure, we impose our second bound on $\eps_2$:
 $$
     \eps_2<\tau_1/K_0
 $$
Thus, by the transversality assumption on $\FF^{cu},\FF^{cs}$, $z$ is well-defined and $x,y\in B(z,KK_0\epsilon)$ .
To use the respect of the domination, we impose our third bound:
 $$
    \eps_2< (2/KK_0)^2,
 $$
so $\rho>KK_0\epsilon$ and therefore, if $R(x,y)<\infty$, then $R(g^nx,g^ny)\to0$ when $n\to\infty$.
The invariance of $\mu$ implies $R(x,y)=0$ or $\infty$ $\mu$-a.e. 

We claim that $R$ is $\pi$-measurable. Otherwise there would
exist a set of positive $\pi_*\nu$-measure of points $x\in M$,
such that 
$$\begin{array}{llll}
\{(y,z)\in M\times M:R(y,z)=0\}\textrm{ and}\\ 
\{(y,z)\in M\times M:R(y,z)=\infty\}
\end{array}
$$ 
have both positive $\nu_x\times\nu_x$-measure.
Now, observe that
 $$
    \{(y,z):R(y,z)=0\} = \{(y,z): \FF^{u}_{\tau_2}(y)=\FF^{u}_{\tau_2}(z)\}
 $$
and that if this set has positive $\nu_x\times\nu_x$-measure for a set of $x\in M$ with positive $\pi_*\nu$-measure, then there exists a measurable function of $x$, $y_x$ such that $\nu_x(\FF^{cu}_{\tau_2}(y_x))>0$ over a set of positive $\pi_*\nu$-measure. Similarly there exists a measurable $z_x$ such that $\nu_x(\FF^{su}_{\tau_2}(z_x))>0$. It follows that:
  $$
    (\nu_x\times\nu_x)(\FF^{cu}_{\tau_2}(y_x)\times\FF^{cs}_{\tau_2}(z_x))>0.
  $$
As $\nu_x\ne\delta_x$ by assumption, it follows that $0<R(y,z)<\infty$
with positive $\mu$-measure, a contradiction.
\end{proof}

\section{Non-concentration}\label{sec:non-concentration}

We consider an asymptotically entropy-expansive diffeomorphism (whose definition is recalled below) and show that its large entropy measures cannot be concentrated around a fixed number of points. We will apply this to Anosov diffeomorphisms.

We recall Bowen's entropy formula for a subset $Y\subset M$ in terms of dynamical $(\eps,n)$-balls $$B_f(x,\eps,n):=\{y\in M:\forall 0\leq k<n\; d(f^ky,f^kx)<\eps\}.$$ We have
 $
    h_\top(f,Y) := \lim_{\eps\to0} h_\top(f,Y,\eps)$ 
 ($h_\top(f)=h_\top(f,M)$)  with
$$
  h_\top(f,Y,\eps):=\limsup_{n\to\infty}\frac1n\log r_f(\eps,n,Y)
 $$
where $r_f(\eps,n,Y)$ is the minimal number of dynamical $(\eps,n)$-balls needed to cover $Y$.
Katok \cite{Katok} established a similar formula for the entropy of an ergodic invariant probability measure:
 $$
     h(f,\mu) = \lim_{\eps\to0} h_\top(f,\mu,\eps) \text{ with } h_\top(f,\mu,\eps):=\limsup_{n\to\infty}\frac1n\log r_f(\eps,n,\mu)
 $$
where $r_f(\eps,n,\mu)$ is the minimal number of dynamical $(\eps,n)$-balls with union of measure at least $1/2$. \footnote{One can replace $1/2$ by  any other fixed number in $(0,1)$.}

Finally, we recall Misiurewicz's local  (or conditional, or tail) entropy \cite{MisiurewiczLocal}:
 $$
    h_\loc(f) := \lim_{\eps\to0} h_\loc(f,\eps) \text{ with }
   h_\loc(f,\eps):= \sup_{x\in M} h_\top(f,B_f(x,\eps,\infty)).
 $$


\subsection{Large entropy measures of $f$}

\begin{lem}\label{lem:deconcentrate-f}
Let $f$ be a homeomorphism of $M$ which is asymptotically $h$-expansive (i.e., $h_\loc(f)=0$) with $h_\top(f)>0$.
For any $\eta>0$ and $N\geq1$, there exist $h<h_\top(f)$ and $r>0$ such that any $\mu\in\Prob_\erg^h(f)$ satisfies $\mu\left(\bigcup_{i=1}^N B(x_i,r)\right)<\eta$, for any set of $N$ points $x_1,\dots,x_N\in M$.
\end{lem}

\begin{proof}
Let $0<\eta<1$. Pick $0<\eps<\eta h_\top(f)/4$. As $f$ is asymptotically $h$-expansive we know
there exists a constant $s_0>0$ such that, for any $s>0$, any subset $Y$, and any $n\geq0$ the following holds
 $$
    r(s,n,Y) \leq C(s,s_0) e^{\eps n} r(s_0,n,Y).
 $$
Also $r(s_0/3,n,M)\leq C_0 e^{(h_\top(f)+\eps)n}$ for some $C_0<\infty$ and all
$n\geq0$.

Observe that, for any $s>0$, any integer $n\geq0$ and a decomposition $n=n_1+\dots+n_k$ into a sum of positive integers
we have 
$$
   r(3s,n,Y) \leq \prod_{i=1}^{k} r(s,n_i,f^{n_1+\dots+n_{i-1}}Y).
 $$ 
To see this, consider the map $\iota:x\mapsto (y_1,\dots,y_{k})$ where  
 $$
    f^{n_1+\dots+n_{i-1}}(x)\in B_f(y_i,s,n_i)
 $$ 
with the $y_i$'s taken from a minimal set of centers of $(s,n)$-balls making a cover of $f^{n_1+\dots+n_{i-1}}Y$. Take a minimal set $C_1$ such that the $\{B_f(x,s,n)\}_{x\in C_1}$ is a cover of $Y$. Select a minimal subset $C_2\subset C_1$ such that $\iota:C_2\to\iota(C_1)$ is a bijection. Clearly the cardinality of $C_2$ satisfies the above bound. We claim that $\{B_f(x,3s,n)\}_{x\in C_2}$ is a cover of $Y$. This follows from the fact that $\iota(x')=\iota(x)$ implies $B_f(x',s,n)\subset B_f(x,3s,n)$.

Now let $n_0<\infty$ satisfy 
$$\frac{\log N+\log C_0}{n_0}<\eps\textrm{ and }\binom{2[n/n_0]+2}{n}\leq e^{\eps n/2}$$
for all $n\geq 0$. Let $r>0$ be such that $B(x,r)\subset B_f(x,r_0/2,n_0)$.

Fix $N$ points $x_1,\dots,x_N$ and $B_r:=B(x_1,r)\cup\dots\cup B(x_N,r)$. Let $\mu$ be an invariant ergodic measure of $f$ with $\mu(B_r)>\eta$.  

We now bound the entropy of $\mu$ by estimating the number of $(s_0,n)$-balls necessary to cover some set $M'$ of measure more than $1/2$.
Observe that for a typical $x$ and $n$ large enough, we can decompose the integer interval
$[0,n[$ into subintervals, half of them being of the form $[a,a+n_0[$ with $f^ax\in
B(x_i,r)$ and the sum of their lengths at least $\eta n$. Therefore, we have
 $$
    r(s_0,n,M') \leq \sum_{n_1+\dots+n_k+kn_0=n} \prod_{i=1}^k r(s_0/3,n_i,M) r(s_0/3,n_0,B_r)
 $$
The previous estimates and the Birkhoff ergodic theorem yield a subset $M'$ of $M$ with $\mu(M')>1/2$ such that,
for all large $n$:
 $$\begin{array}{llll}
     \frac1n\log r(s_0,n,M') \leq &\eta \log N/n_0 + (1-\eta) (h_\top(f)+\eps) + \\
     & \log C_0/n_0
    + \log \binom{2[n/n_0]+1}{n}/n.
    \end{array}
 $$ 
It follows that
 $$\begin{aligned}
    h(f,\mu) &\leq h(f,\mu,s_0)+\eps \\
    & \leq (\eta \log N+\log C_0)/n_0  +  (1-\eta) (h_\top(f)+\eps) +2\eps \\
     & \leq h:=h_\top(f)+3\eps -\eta h_\top(f)
    < h_\top(f).
 \end{aligned}$$
\end{proof}

\section{Proof of Theorem~\ref{thm:main}}\label{sec:factor}

Let $N\geq1$ be an integer and let $f$ be an Anosov diffeomorphism of a compact manifold $M$.

\subsection{Choice of the numbers $\alpha,\eps>0$ and $\Lambda>1$}

We endow $M$ with an adapted Riemannian metric. Let $\eps_*,K_*$ be the two numbers as in the shadowing Lemma (Lemma \ref{l.shadowing}).  Let $\lambda>1$ be the hyperbolicity strength of $f$.

We fix $\Lambda\in (1,\lambda)$ and pick $\alpha>0$ and $0<\eps_0<\eps_*/2$ small enough so that $\Lambda>\frac{\epsilon_0^{1/2} +2\epsilon_0}{\epsilon_0^{1/2}-2\epsilon_0}>1$ and, for all $\tilde f\in\Diff^1(M)$ with $d_{C^1}(\tilde f,f)<\eps_0$, for all $x\in M$:
 $$\begin{aligned}
    &\forall v\in C^u_{\alpha}(x)\qquad \| D\tilde fv\|\geq\Lambda\|v\|  \text{ and }D\tilde f v\in C^u_\alpha(\tilde fx) \\
    &\forall v\in C^s_{\alpha}(x)\qquad \| D\tilde fv\|\leq\Lambda^{-1}\|v\|  \text{ and }D\tilde f^{-1} v\in C^s_\alpha(\tilde f^{-1}x)
 \end{aligned}$$
where $C^u_\alpha,C^s_\alpha$ are the cone fields with aperture $\alpha$ around the unstable and stable bundles of $f$. We also fix $R_0>0$ such that, for all $x\in M$, for all $y\in (x+C^u_{\alpha}(x))\cap B(x,R_0)$
 \begin{equation}\label{eq:exp-Anosov1}
    \| \tilde fy-\tilde fx\|\geq\Lambda\|y-x\| 
              \text{ and }\tilde f y-\tilde f x\in C^u_\alpha(\tilde f x)\\
 \end{equation}
and, likewise, if $y\in (x+C^s_{\alpha}(x))\cap B(x,R_0)$
 \begin{equation}\label{eq:exp-Anosov2}
      \| \tilde f^{-1}y-\tilde f^{-1}x\|\geq\Lambda\|y-x\| 
              \text{ and }\tilde f^{-1}y-\tilde f^{-1} x\in C^s_\alpha(\tilde f^{-1}x).
 \end{equation}

Observe that by compactness of $M$ and transversality of the cone fields, there are constants $\tau_1,\tau_2>0$ and $K<\infty$ such that, any pair of continuous foliations $\FF^1,\FF^2$ tangent to $C^u_\alpha,C^s_\alpha$ have a product structure with these constants. 

We fix $\eta>0$ small enough so that $\Lambda^{1-\eta}e^{-\eta}>1$. $\eta>0$ and $N\geq1$ being fixed, Lemma \ref{lem:deconcentrate-f} yields two numbers $h_0<h_\top(f)$ and $r_0>0$ such that for any $\mu\in\Prob_\erg^{h_0}(f)$, $\mu(B_{r_0})<\eta$. We fix $\gamma>0$ so small that $h_1:=h_0+d\gamma<h_\top(f)$.

We reduce $\eps_0$ so that $\eps_0>0$ and $\eps_0$ is less than the following:
\begin{itemize}
 \item $\eps_1(f,\Lambda,\alpha)$,
 \item $\eps_2(f,\Lambda,\alpha,\tau_1,\tau_2,K)$,
 \item $\tau_2/K_*$,
 \item $r_0/(1+2K_*+KK_*)$, and
 \item $R_0/KK_*,1/(KK_*+2)^2.$
 \end{itemize} 
where $\eps_1,\eps_2$ have been defined in Theorem \ref{thm:foliations} and in Proposition \ref{prop:flat}. 

\subsection{Entropy decrease under $\pi$}
We show that measures with large entropy for $g$ project to measures with large entropy for $f$.

As $0<\eps<\eps_1$, Theorem~\ref{thm:foliations} yields $g$-invariant center-unstable and center-stable foliations $\FF^{cu},\FF^{cs}$ with $C^1$ leaves. Recall that the shadowing Lemma defines a factor map $\pi:M\rightarrow M$ with 
$$\diam(\pi^{-1}(\pi(x)))\leq K_*d_{C^0}(g,f)<K_*\eps < \tau_2.$$ 

Let $\nu\in\Prob_\erg^{h_1}(g)$. As $0<\eps<\eps_2$, Proposition~\ref{prop:flat}  gives a set $X\subset M$ with $\nu(X)=1$ and $\sigma=cu$ or $cs$, such that $\pi^{-1}(\pi(x))\cap X\subset\FF^\sigma_{\tau_2}(x)$ for $\nu$-a.e. $x\in M$. We assume that $\sigma=cu$ and leave the similar case $\sigma=cs$ to the reader.  

 Let $\mu:=\pi_*(\nu)\in\Prob_\erg(f)$.  Recall that inverting a transformation does not change its measure-theoretic entropy so we have the following (easy extension of the) Ledrappier-Walters~\cite{LW} inequality:
 $$
  h(g^{-1},\nu) \leq h(f^{-1},\mu) + \int_M h_\top(g^{-1},\pi^{-1}(\pi(x))\cap\FF^{cu}(x)) \, \nu(dx).
 $$
   The dilation under $g^{-1}$ of the center-unstable
leaves is bounded by $e^{\gamma}$, so $h_\top(g^{-1},\FF^{cu}_\delta(x))\leq \dim \FF^{cu}\cdot \gamma$.
 It follows that
 \begin{equation}\label{eq:loss-entropy}
  h(f,\mu) \geq h(g,\nu)-d\gamma > h_0.
 \end{equation}

\subsection{Entropy-conjugacy}

We let $0<\eps<\eps_0$ and pick $g\in\Diff^1(M)$ satisfying (H1)-(H3) from Theorem~\ref{thm:main}. Let
 $$
    M':=\{x\in M:\pi^{-1}(\pi(x))=\{x\} \}\text{ and }M'':=\pi(M').
 $$
These are measurable subsets. We show that $M'$ and $M''$,  have full measure with respect to any measure in $\Prob_\erg^{h_1}(g)$ and $\Prob_\erg^{h_0}(g)$ respectively, with $h_0,h_1<h_\top(f)\leq h_\top(g)$ defined above.

First, consider $\nu\in\Prob_\erg^{h_1}(g)$.
From \eqref{eq:loss-entropy}, Proposition \ref{lem:deconcentrate-f} yields $\pi_*(\nu)(B_{r_0})<\eta$. But 
 $$
     \pi^{-1}(B_{r_0}) \supset \bigcup_{i=1}^N B(x_i,r_0-2K_*\eps_0)\supset B_{r+KK_*\eps}.
 $$
Indeed $r+KK_*\eps\leq \eps+KK_*\eps<r_0-2K_*\eps$. It follows that
$$\nu(B_{r+KK_*\eps})\leq \nu(\pi^{-1}(V_{r_0}))=\mu(V_{r_0})<\eta.$$

Let $x$ be a $\nu$-typical point and let $y\in \pi^{-1}(\pi(x))$. Note that $d(x,y)<K_*\eps<\tau_1$, hence the following points are well-defined: $y^s:=\FF_{\tau_2}^{cs}(x)\cap\FF_{\tau_2}^{cu}(y)$ and $y^u:=\FF_{\tau_2}^{cu}(x)\cap\FF^{cs}_{\tau_2}(y)$. The transversality of $\FF^{cu}$ and $\FF^{cs}$ implies that $d(x,y^s)\leq K d(x,y)\leq KK_*\eps$ and, likewise, $d(x,y^u)\leq KK_*\eps$.

As $g$ is $\gamma$-nearly hyperbolic and $y^u\in\FF^{cu}_{\tau_2}(x)$ we have
 $$
    d(g(x),g(y^u))\geq e^{-\gamma} d(x,y^u).
 $$
As $g$ respects the domination of $f$ and $KK_*\eps<\rho:=(\eps^{1/2}-2\eps)$ we have $g(y^u)-g(x)\in C^u_\alpha(g(x))$. 

Consider now the special case where $x\notin B_{r+KK_*\eps}$. Then $d(x,y^u)<KK_*\eps<R_0$ and $y^u-x\in C^u_\alpha(x)$. As $x,y^u\notin B_r$ and $d(x,y^u)<R_0$, we can use the estimates \eqref{eq:exp-Anosov1} and \eqref{eq:exp-Anosov2} and obtain the better lower bound
 $$
   d(g(x),g(y^u))\geq \Lambda d(x,y^u).
 $$

Define $m:M\to\RR$ by $m(x)=\Lambda^{-1}$ if $x\notin B_{r+KK_*\eps}$ and $m(x)=e^{-\gamma}$ otherwise. An induction yields:
 $$
    \forall n\geq0\quad d(g^n(y^u),g^n(x)) \geq \prod_{k=0}^{n-1} m(g^k(x)) d(x,y^u)
 $$
and Birkhoff Ergodic Theorem implies, for $\nu$-a.e. $x\in M$ that
 $$
      \lim_{n\to\infty} (1/n)\sum_{k=0}^{n-1} \log m(g^kx) \geq (1-\eta)\log\Lambda-\eta\gamma > 0.
 $$
As $d(g^n(y^u),g^n(x))\leq KK_*\eps$ for all $n$, we must have $x=y^u$. Likewise $x=y^s$. Thus, $x=y$ a.e. and $\nu(M')=1$.

Let $\mu\in\Prob_\erg^{h_0}(f)$. Proposition \ref{lem:deconcentrate-f} directly shows $\mu(V_{r_0})<\eta$. By compactness, there exists $\nu\in\Prob(g)$ with $\pi_*(\nu)=\mu$ and we can conclude as above that, for $\mu$-a.e. $x$, $\pi^{-1}(x)$ is a single point: $\mu(M'')=1$.

\subsection{Almost expansivity}

In the  previous section, we showed that $\pi(y)=\pi(x)$ implies $y=x$ for $\nu$-a.e. $x\in M$ and all $y\in M$ whenever $\nu\in\Prob_\erg^{h_1}(g)$. The hypothesis $\pi(x)=\pi(y)$ was only used to show that $\sup_{n\in\ZZ} d(g^nx,g^ny)< K_*\eps$. Hence, the above reasoning implies that $K_*\eps>0$ is an expansivity constant with respect to all large entropy measures of $g$. This finishes the proof of Theorem \ref{thm:main}.

\section{Proof of Theorem \ref{thm:applyToBV}}\label{s.BV}

In this section we prove that there is a $C^1$-entropically stable, $C^1$-robustly transitive diffeomorphism $g$ of the 4-torus which is not partially hyperbolic. More precisely, we check that the Bonatti-Viana example of a non-partially hyperbolic, robustly transitive diffeomorphism satisfies the assumptions of Theorem \ref{thm:main}.  Then, we show that the construction can be modified to obtain arbitrarily large symbolic extension entropy $h_{\operatorname{sex}}(g)$ as stated at the end of Theorem \ref{thm:applyToBV}.

Let $A$ be a $4$ by $4$ matrix with integer entries and determinant one with four distinct real eigenvalues where
 $$
      0<\lambda_1< \lambda_2 < 1/3<3<\lambda_3< \lambda_4
$$
and that the induced hyperbolic toral automorphism, $f_A$, on the 4-torus has at least 4 fixed points, say $p,q,r,s$. 
Following Bonatti and Viana \cite{BV},  one of the fixed points, say $s$, will be left alone to ensure the robust transitivity. A deformation will be done around two others, say $p$, respectively $q$, to forbid hyperbolicity and the existence of any invariant subbundle of the central-stable, respectively central-unstable, subbundle.
The last point $r$ will be used to obtain diffeomorphism with no symbolic extension using techniques from \cite{BD}. We must check that this construction can be performed under the assumptions of Theorem \ref{thm:main} for $N=3$ and $\alpha,\gamma,\eps$ small enough, i.e., smaller than $t(f,N)$ and $\Lambda>\frac{\eps^{1/2}+2\eps}{\eps^{1/2}-2\eps}$.

\begin{figure}[htb]
\begin{center}
\psfrag{A}{$f_A$}
\psfrag{B}{$$}
\psfrag{q}{$q$}
\psfrag{r}{$q_1$}
\psfrag{s}{$q$}
\psfrag{t}{$q_2$}
\psfrag{u}{$q_1$}
\psfrag{v}{$q$}
\psfrag{w}{$q_2$}
\psfrag{C}{$f_0$}
\includegraphics[width=5in]{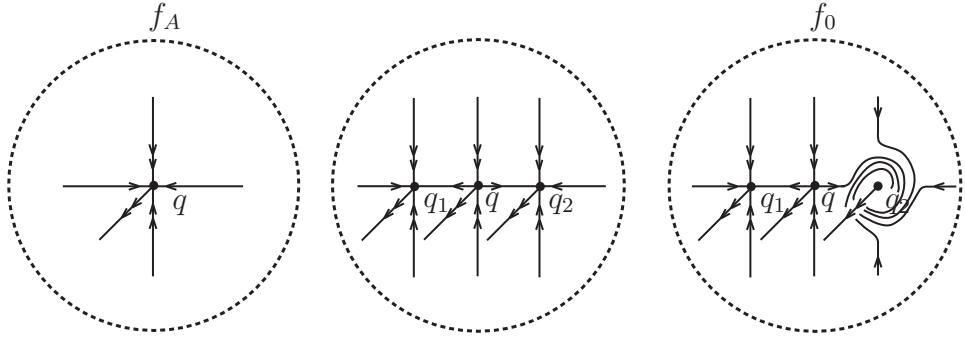}
\caption{Bonatti Viana construction}\label{f.bv}
\end{center}
\end{figure}

We begin by deforming $f_A$ around the two fixed points $p,q$.  Fix $\eta>0$ so small that $\lambda_3-\eta>1$. Let $\gamma\in(0,t)$ small enough so that
 $$
   (\lambda_3-\eta)/e^\gamma>1.
 $$
 
We fix $\epsilon\in (0,t)$ such that the balls of radius $2\epsilon$ around $p,q,r,s$ are disjoint and
$$1<\frac{\epsilon^{1/2} +2\epsilon}{\epsilon^{1/2} -2\epsilon}<\frac{\lambda_3-\eta}{e^{\gamma}}.$$  

We deform $f_A$ into $f_0$ inside $B(q,\epsilon/2)$ 
keeping $\mathcal F_A^u$ invariant. We do this in two steps. 
In the first step, we do a pitchfork bifurcation around $q$ in the stable direction $\lambda_2$. 
The stable index of $q$ changes from 2 to 1 and two new fixed points $q_1$ and $q_2$ are created.  Then we perturb the diffeomorphism in a neighborhood of $q_2$ so that the contracting eigenvalues become complex; see Figure~\ref{f.bv}.

To be more precise, let $D^2$ be the two dimensional disk and $\phi:D^2\times D^2\rightarrow \mathbb{T}^4$ be a linear chart mapping
\begin{itemize}
\item  0 to $q$, 
\item disks $D^2\times\{y\}$, $D_xg_{s_1}^{-1}(C^{cs}_\beta)\subset C^{cs}_\beta$ into the stable leaves of $f_A$, and 
\item disks $\{x\}\times D^2$, into the unstable leaves of $f_A$.
\end{itemize}
Let $\chi:D^2\rightarrow [0,1]$ be a smooth cutoff function: so $\chi(0)=1$ and $\chi$ is 0 in a neighborhood of the boundary of $D^2$.

Let $\Psi$ be a volume preserving vector field on $D^2$ such that $\Psi$ has a saddle singularity at the origin with one axis, $e_2$, being expanding and the other, $e_1$, contracting,and $\Psi$  is zero in a neighborhood of the boundary of $D^2$.  Let 
 $
    \tilde{\Psi}(x,y)=(\chi(y)\Psi(x), 0).
 $
We denote by $\phi_*\tilde{\Psi}$ the push-forward and by $(\Psi)_a$ the time $a$ of the flow defined by a vector field $\Psi$. Let
 $$
    f_{A,a}=(\phi_*\tilde{\Psi})_a\circ f_A
$$
Observe that the point $q$ remains fixed for all $a\geq 0$ and that the weakest contracting eigenvalue, $\lambda_2=\lambda_2(q,a)$, of $Df_{A,a}(q)$ increases as $a$ increases. It is easy to arrange it so that the expansion at other points is not stronger than that at $q$.  So there exists some $a_0>0$ such that the eigenvalue in the direction $e_2$ is 1 for $Df_{A,a_0}(q)$.  For $a>a_0$ we have expansion in this $e_2$ direction.  Fix  $a_1$ larger, but sufficiently close, to $a_0$ such that $\lambda_2(q,a_1)\leq e^{\gamma/2}$.  Note that $f_{A,a_1}$ is $\gamma/2$-nearly hyperbolic.

We let $g_0=f_{A,a_1}$ and perturb $g_0$ in a neighborhood of $q_2$ that is disjoint from $q$, using a similar, smaller chart.  Let $\Phi$ be a volume preserving vector field of $D^2$ that is zero in a neighborhood of the boundary of $D^2$ and defines a fixed point of center type at the origin.  Let
 $
   \tilde{\Phi}(x,y)=(\chi(y)\Phi(x), 0)
 $
and
 $$
   g_b=(\phi_*\tilde{\Phi})_b\circ g_0.
 $$ 
For some $b_0>0$, the two contracting eigenvalues of $q_2$ for $Dg_{b_0}(q_2)$ become equal.  For $b_1$ slightly larger, these eigenvalues are (non-real) complex conjugates.

Note also that the creation of fixed points with different indices prevents
the topologically transitive map from being Anosov.
These non-real eigenvalues also forbid the existence of a one-dimensional invariant sub-bundle inside $E^{cs}$.

The differential of $g_{b_1}$ at each point of $M$ has the following form, using block matrices, in the eigenbasis $(v_1,v_2,v_3,v_4)$ (which we can and do assume to be orthonormal)
 $$
    \left(\begin{matrix}
            \Lambda_{cs} & K \\
            0            & \Lambda_u
            \end{matrix}\right) \text{ where }
    \Lambda_u = \left(\begin{matrix}
            \lambda_3 & 0 \\
            0         & \lambda_4
            \end{matrix}\right)
 $$
and $\Lambda_{cs}$ and $K$ are (variable) $2$-by-$2$ matrices with 
 $$
\|\Lambda_{cs}\|:=\sup_{x\in M}\sup_{\|v\|=1} \|\Lambda_{cs}(v) \|\leq e^{\gamma/2}.
 $$

The stable foliation for $f_A$ is invariant under $g_{b_1}$, even though its tangent vectors are not necessarily contracted under $Dg_{b_1}$. Thus, any thin cone field $C^s_\alpha$ will be invariant under $g_{b_1}^{-1}$.
More specifically, the inverse of the above matrix is
$$\left(\begin{matrix}
            \Lambda_{cs}^{-1} & -\Lambda_{cs}^{-1}K\Lambda_u^{-1} \\
            0            & \Lambda_u^{-1}
            \end{matrix}\right).
$$
 
On the one hand, fixing $\alpha\in (0,t)$ small enough so that 
\begin{equation}\label{eq:invcon-cs1}
    \alpha < \frac{\lambda_3-e^{\gamma/2}}{\|K\|}
\end{equation}
ensures the invariance $D_xg_{s_1}^{-1}(C^{s}_\alpha)\subset C^{cs}_\alpha$
and such that for all  non-zero vectors $v\in C^s_{\alpha}$ and $w\in C^u_{\alpha}$, 
$$\|Df v\| <  e^\gamma\|v\| \text{ and } \|Dfw\|> (\lambda_3-\eta)\|w\|.
$$
Recall that the cones are defined using the invariant splitting of the original map $f_A$.

On the other hand, the vectors in the unstable subbundle for $f_A$ are still expanded by $Dg_{b_1}$, but the subbundle is no longer invariant.  
  Let 
$C$ be the complement of the center-stable conefield, i.e., $C(x):=\overline{T_xM\setminus C^s_\alpha(x)}$.  Then $C$ is an invariant strong-unstable cone field for $g_{s_1}$, but usually very wide.  To rectify this, we modify $g_{b_1}$ in $B(q, \epsilon)$.
One defines $f_1$ around $p$ and $q$ by 
$$f_1:=L\circ g_{s_1}\circ L^{-1}$$ 
 where 
 $$
    L = \left(\begin{matrix}
            \alpha^2 & 0  & 0 &0\\
            0         & \alpha^2 & 0 & 0\\
            0 & 0 & 1 & 0 \\
            0 & 0 & 0 &1
            \end{matrix}\right)
 $$
We set $f_1=g_{s_1}$ elsewhere. This yields a diffeomorphism since $f_A=L\circ f_A\circ L^{-1}$. $C^{u}_{\alpha}$ is mapped to $C$ by $L^{-1}$, so is an invariant cone field for $f_1$.  Also $L^{-1}(C^s_\alpha)\subset C^s_\alpha$. So $f_1$ preserves the two cone fields $C^s_\alpha$ and $C^u_\alpha$ and is $\gamma$-nearly hyperbolic.
 
We explain why $f_1$ ($\alpha, \rho,\Lambda)$-respects the domination for $f_A$ where $\rho=(\epsilon^{1/2}-2\epsilon)$ and $\Lambda=(\lambda_3-\eta )/e^\gamma$.
 
 The cones $C^u_\alpha(x)$ and $C^s_\alpha(x)$ are constant. Hence,  for all $x\in M$ and $y,z\in B(x, \rho)$,  if $y-x\in C^u_\alpha(x)$ and $f_1(z)-f_1(x)\in C^s_{\alpha}(f_1(x))$, then $f_1(y)-f_1(x)\in C^u_{\alpha}(f_1(x))$ and $z-x\in C^s_{\alpha}(x)$. Moreover,
 $$\frac{\|f_1(y)-f_1(x)\|}{\|f_1(z)-f_1(x)\|}\geq \frac{(\lambda_3-\eta)(\|y-x\|}{e^\gamma\|z-x\|}.$$
 Hence,
 $$\frac{\|f_1(y)-f_1(x)\|}{\|y-x\|}>\frac{\lambda_3-\eta}{e^\gamma}\frac{\|f_1(z)-f_1(x)\|}{\|z-x\|}=
 \Lambda\frac{\|f_1(z)-f_1(x)\|}{\|z-x\|}.$$
 Hence, the map $f_1$ ($\alpha, \rho,\Lambda)$-respects the domination for $f_A$. $f_1$ is clearly an $(\eps,3)$-sparse deformation of $f_A$ and  we noticed that it is $\gamma$-nearly hyperbolic.

To finish the construction we repeat the deformation just made on $f_A$ near $p$, on $f_1^{-1}$ in the neighborhood of radius $\epsilon$ around $q$. We get a map $f$ which is robustly transitive,  not partially hyperbolic, and has a dominated splitting $T\mathbb T^4=E^{cs}\oplus E^{cu}$ with 
$\dim E^{cs}=\dim E^{cu}=2$ (see~\cite{BV} for proofs of these facts).  Furthermore, by construction the map $f$ satisfies the hypothesis of Theorem~\ref{thm:main} and so is entropy conjugate to $f_A$ and $C^1$-entropically stable, proving the first half of Theorem~\ref{thm:applyToBV}.

\section{Symbolic Extensions}\label{s.symbext}

We modify the diffeomorphism $f$ constructed in Sec.~\ref{s.BV} so that there is no symbolic extension, as stated in Theorem \ref{thm:applyToBV}.  The deformation will be done around the fourth fixed point, $r$, for the diffeomorphism $f$.  The construction will closely follow the methods in \cite{DN} and some of the discussion in \cite{DF}.  

\begin{figure}[htb]
\begin{center}
\psfrag{A}{$f$}
\psfrag{B}{$$}
\psfrag{q}{$r$}
\psfrag{s}{$r$}
\psfrag{t}{$r$}
\psfrag{C}{$g_0$}
\includegraphics[width=5in]{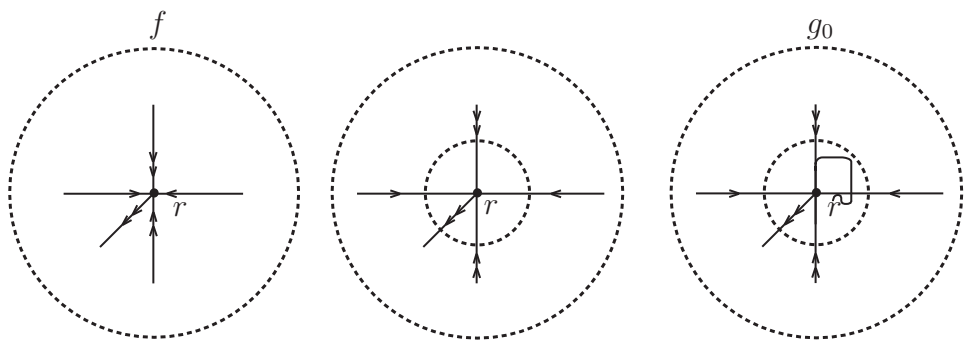}
\caption{Homoclinic tangency for $r$}\label{f.tangency}
\end{center}
\end{figure}

First we fix $\epsilon>0$ small enough for Theorem \ref{thm:main}  and modify $f$ in the ball of radius $\epsilon$ centered at $r$, along the center-stable direction, just as in the first step of the deformation about $p$.  However, 
we pick the parameter $a_1$ so that the differential becomes the identity along the center-stable direction at $r$. We further modify $f$ so that, not only the differential, but the map itself restricted to the center-stable leaf of $r$ is the identity  in a small ball $B(r,\tau)$. Now we perturb to obtain a new map $g_0$ such that $r$ is a saddle fixed point in the stable direction with directions that are slightly expanding and contracting (without violating the $\gamma$-near hyperbolicity) and such that $r$ has a homoclinic tangency inside $B(r, \tau)$ in the stable leaf.
  See Figure \ref{f.tangency}.  As in the previous arguments we can do this in such a way that the deformed map $g_0$ will satisfy the conditions of Theorem \ref{thm:main}.

\begin{figure}[htb]
\begin{center}
\psfrag{D}{$g_1$}
\psfrag{B}{$$}
\psfrag{q}{$r$}
\psfrag{s}{$r$}

\psfrag{v}{$r$}
\psfrag{C}{$g_0$}
\includegraphics[width=5in]{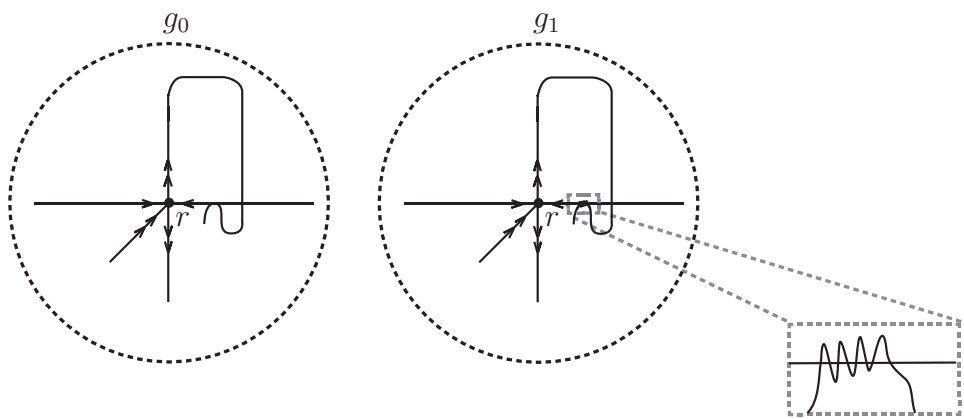}
\caption{Creation of a horseshoe}\label{f.tangency2}
\end{center}
\end{figure}

  We now perturb $g_0$ to obtain a map $g$ with a neighborhood $\mathcal{V}\subset \mathrm{Diff}^1(\mathbb{T}^4)$ and a $C^1$-residual set $\mathcal{D}\subset \mathcal{V}$ such that each $\tilde{g}\in \mathcal{D}$ has no symbolic extension.
 We first perturb in a  $C^1$ small, but $C^2$ large manner. 
The idea is to create a number of transverse intersections near the pervious homoclinic point.  See Figure~\ref{f.tangency2}.  From these transverse homoclinic points we obtain a locally maximal hyperbolic set with topological entropy larger than some constant.  This is now the situation examined in Downarowicz and Newhouse~\cite{DN} where they show the nonexistence of symbolic extensions.  The only difference is that we are working on a 2-dimensional leaf of a foliation whereas they are dealing with surfaces.  A detailed explanation of this procedure is given in~\cite{DF}.  So there exists an open set $\mathcal{V}$ in $\mathrm{Diff}^1(\mathbb{T}^4)$ such that each $g\in \mathcal{V}$ is robustly transitive, not partially hyperbolic and entropically conjugate to $f_A$.  Furthermore, there is a $C^1$-residual set $\mathcal{D}$ in $\mathcal{V}$ such that each diffeomorphism in $\mathcal{D}$ has no symbolic extension.  This concludes the proof of Theorem~\ref{thm:applyToBV}.

\end{document}